\documentclass[11pt]{amsart}

\usepackage{amsmath,amscd}
\usepackage{amssymb}
\usepackage{amsthm}

\usepackage{graphicx}

\usepackage{mathrsfs}
\usepackage[ocgcolorlinks, linkcolor=blue]{hyperref}

\usepackage{calc}
             {\begin{list}{\arabic{enumi}.}{\usecounter{enumi}%
              \setlength{\labelsep}{0.5em}%
              \settowidth{\labelwidth}{(\arabic{enumi})}%
              \setlength{\leftmargin}{\labelwidth+\labelsep}}}%
             {\end{list}}

\usepackage{comment}

\DeclareRobustCommand{\SkipTocEntry}[5]{}

\usepackage{xcolor}
\definecolor{LOcolor}{RGB}{150,100,0}

\newtheorem{Theorem}{Theorem}[section]

\newtheorem{Corollary}[Theorem]{Corollary}

\newtheorem{Proposition}[Theorem]{Proposition}

\newtheorem{Question}[Theorem]{Question}

\theoremstyle{definition}
\newtheorem{Definition}[Theorem]{Definition}

\newtheorem{Example}[Theorem]{Example}

\newtheorem{Remark}[Theorem]{Remark}

\numberwithin{equation}{section}

\newcommand{\mR}{\mathbb{R}}                    
\newcommand{\mC}{\mathbb{C}}                    
\newcommand{\mN}{\mathbb{N}}                    
\newcommand{\abs}[1]{\lvert #1 \rvert}          
\newcommand{\norm}[1]{\lVert #1 \rVert}         
\newcommand{\R}{\mathbb{R}}   

\newcommand{\ol}[1]{\overline{#1}}

\newcommand{\eps}{\varepsilon}

\newcommand{\dbar}{\overline{\partial}}

\newcommand{\p}{\partial}

\newcommand{\s}{\hspace{0.5pt}}

\newcommand{\M}{\widetilde{M}}

\def\C{\mathbb C}
\def\N{\mathbb N}

\newcounter{sidenote}

\begin{document}

\title{Applications of the Stone-Weierstrass theorem in the Calder\'on problem} 

\author[T. Liimatainen]{Tony Liimatainen}
\address{Department of Mathematics and Statistics, University of Helsinki, PO Box 68, 00014 Helsinki, Finland}
\email{tony.liimatainen@helsinki.fi}

\author[M. Salo]{Mikko Salo}
\address{Department of Mathematics and Statistics, University of Jyv\"askyl\"a, PO Box 35, 40014 Jyv\"askyl\"a, Finland}
\email{mikko.j.salo@jyu.fi}




\begin{abstract}
We give examples on the use of the Stone-Weierstrass theorem in inverse problems. We show uniqueness in the linearized Calder\'on problem on holomorphically separable K\"ahler manifolds, and in the Calder\'on problem for nonlinear equations on conformally transversally anisotropic manifolds. We also study the holomorphic separability condition in terms of plurisubharmonic functions. The Stone-Weierstrass theorem allows us to generalize and simplify earlier results. It also makes it possible to circumvent the use of complex geometrical optics solutions and inversion of explicit transforms in certain cases.
\end{abstract}

\maketitle

\section{Introduction} \label{sec_introduction}

In this work we study versions of the geometric (or anisotropic) Calder\'on problem. This inverse problem was studied in \cite{Calderon1980} for the purpose of determining the electrical conductivity in a Euclidean domain from voltage and current measurements on its boundary. There is a substantial literature on the Calder\'on problem and we refer the readers to the survey \cite{Uhlmann2014}.

The anisotropic Calder\'on problem corresponds to the case of matrix-valued conductivities. In dimensions $\geq 3$ the problem can be reformulated in geometric terms as follows (see e.g.\ \cite{DosSantosFerreira2009} for further details). Let $(M,g)$ be a compact oriented Riemannian manifold with smooth boundary, and let $q \in C^{\infty}(M)$. Consider the Cauchy data set 
\[
C_{g,q} = \{ (u|_{\p M}, \p_{\nu} u|_{\p M}) \,:\, \text{$u \in C^{\infty}(M)$ satisfies $(-\Delta_g+q)u = 0$ in $M$} \}.
\]
Here $\Delta_g$ is the Laplace-Beltrami operator and $\p_{\nu}$ is the normal derivative with respect to $g$. The anisotropic Calder\'on problem corresponds to determining a Riemannian metric $g$, up to a boundary fixing isometry, from the knowledge of $C_{g,0}$. For metrics in a fixed conformal class the problem reduces to determining an unknown potential $q$ from the knowledge of $C_{g,q}$. In this work we will consider the case where $g$ is fixed and we wish to recover $q$. A standard method for proving such results is to show a density result of the following type.

\begin{Question}[Completeness of products]
Let $(M,g)$ be a compact Riemannian manifold with smooth boundary, and let $q_1, q_2 \in C^{\infty}(M)$. If $f \in C^{\infty}(M)$ satisfies 
\[
\int_M f u_1 u_2 \,dV = 0
\]
for any $u_1, u_2 \in C^{\infty}(M)$ satisfying $(\Delta_g + q_j)u_j = 0$ in $M$, is it true that $f \equiv 0$?
\end{Question}

In other words, one would like to prove that the set $\{ u_1 u_2 \,:\, (\Delta_g + q_j)u_j = 0 \}$ is complete (i.e.\ its linear span is dense). This is known when $\dim(M) = 2$ \cite{guillarmou2011calderon} but it is an open problem for $\dim(M) \geq 3$. For Euclidean domains this was proved in \cite{sylvester1987global}, and the works \cite{DosSantosFerreira2009, ferreira2013calderon} establish this when $(M,g)$ is conformally transversally anisotropic (CTA, see Definition \ref{def_cta} below)  and the transversal manifold has injective geodesic X-ray transform. See \cite{CFO2023} for related rigidity results and \cite{UhlmannWang, MaSahooSalo} for results at a large fixed frequency.

One can also consider the linearized version of the above inverse problem (linearized at zero potential). This corresponds to completeness of products of harmonic functions:

\begin{Question}[Linearized problem]
Let $(M,g)$ be a compact Riemannian manifold with smooth boundary. If $f \in C^{\infty}(M)$ satisfies 
\[
\int_M f u_1 u_2 \,dV = 0
\]
for any $u_1, u_2 \in C^{\infty}(M)$ satisfying $\Delta_g u_j = 0$ in $M$, is it true that $f \equiv 0$?
\end{Question}

This problem is also open when $\dim(M) \geq 3$ but there are partial results for CTA manifolds with weaker assumptions on the transversal manifold \cite{DKLLS2020, KLS2022}. There is also a result for K\"ahler manifolds \cite{guillarmou2019linearized} based on extending the methods from $\dim(M) = 2$ to higher complex dimensions.

The results above are based on constructing special complex geometrical optics (CGO) solutions to $(\Delta_g+q)u = 0$, and on reducing the uniqueness result to inverting certain transforms. Transforms that have been used in this context include 
\begin{itemize}
\item 
the Fourier transform \cite{sylvester1987global}
\item 
a mixed Fourier/geodesic X-ray transform \cite{DosSantosFerreira2009, ferreira2013calderon}
\item 
a mixed Fourier/FBI transform \cite{DKLLS2020, KLS2022}
\item 
a transform related to stationary phase \cite{guillarmou2019linearized}
\end{itemize}

In this work we introduce an alternative approach where completeness of products is obtained from the \emph{Stone-Weierstrass theorem} rather than from inverting explicit transforms. This approach might work also in the absence of CGO type solutions, and thus it avoids one of the main obstacles in the geometric Calder\'on problem. However, the success of this approach relies on finding suitable algebras within the closure of $\mathrm{span}\{ u_1 u_2 : \text{$u_j$ are solutions} \}$. So far we have only been able to implement this in the presence of suitable complex structure, and for our second main result we also use CGO solutions to exhibit such algebras.

Our first result proves uniqueness in the linearized Calder\'on problem on compact K\"ahler manifolds $M$ with smooth boundary. We assume that $M$ is \emph{holomorphically separable}, i.e.\ for any $x, y \in M$ with $x \neq y$ there is $f \in C^{\infty}(M)$ that is holomorphic in $M^{\mathrm{int}}$ such that $f(x) \neq f(y)$. Previously this uniqueness result was proved in \cite{guillarmou2019linearized} (see \cite{MaTzou2021, KrupchykUhlmannYan} for related results) under the additional condition that $M$ has  local charts given by global holomorphic functions, and for measures $\mu = f \,dV$ where $f \in C^{\infty}(M)$ vanishes to high order at the boundary. The proof was based on CGO solutions with Morse phase functions. Here we remove these additional conditions and give a short proof that avoids CGO solutions and is much simpler than the one in \cite{guillarmou2019linearized}.

\begin{Theorem} \label{thm_kaehler}
Let $(M,g)$ be a compact holomorphically separable K\"ahler manifold with smooth boundary. If $\mu$ is a bounded measure on $M$ (i.e.\ $\mu \in (C(M))^*$) and 
\[
\int_M u_1 u_2 \,d\mu = 0
\]
for any $u_1, u_2 \in C^{\infty}(M)$ with $\Delta_g u_j = 0$ in $M$, then $\mu = 0$.
\end{Theorem}
\begin{proof}
Consider the set 
\[
A = \{ h_1 a_1 + \cdots + h_N a_N \,;\, \text{$h_j \in C^{\infty}(M)$ holomorphic, $a_j \in C^{\infty}(M)$ antiholomorphic, $N \geq 1$} \}.
\]
This set is a subalgebra of $C(M)$ since products of holomorphic (resp.\ antiholomorphic) functions are holomorphic (resp.\ antiholomorphic). Moreover, $A$ is unital and closed under complex conjugation. By assumption $A$ separates points in $M$. By the complex Stone-Weierstrass theorem (see e.g. \cite[Chapter 3, Theorem 1.4]{La93}), $A$ is dense in $C(M)$.

Since holomorphic and antiholomorphic functions on a K\"ahler manifold are harmonic, our assumption on $\mu$ implies that $\int_M w \,d\mu = 0$ for any $w \in A$. Since $A$ is dense in $C(M)$, we obtain $\mu = 0$.
\end{proof}

In Section \ref{sec_kaehler} we show that the holomorphic separability condition (stated for holomorphic functions near $M$) is equivalent to the existence of a smooth strictly plurisubharmonic function on $M$, and that this condition fails for manifolds $M$ such as a compact neighborhood of the equator in $\mC P^n$. We also show that the second assumption in \cite{guillarmou2019linearized} related to local charts given by global holomorphic functions actually follows from a suitable version of holomorphic separability. We remark that harmonic functions always separate points in $M$ e.g.\ by Runge approximation \cite{lassas2018poisson}, so the holomorphic separability assumption is probably just an artifact of our method of proof.

The fact that one has uniqueness in the linearized problem yields uniqueness in inverse problems for nonlinear equations by the method of higher order linearization. Below is an example of such a result (see e.g.\ \cite{SaloTzou2023} for more details). It is likely that one could also consider potentials in $L^p(M)$ for suitable $p$ as in \cite{nurminen2023determining}.

\begin{Corollary}
Let $(M,g)$ be a compact holomorphically separable K\"ahler manifold with smooth boundary, let $q_1, q_2 \in C^{\alpha}(M)$, and let $m \geq 2$ be an integer. Let $\Lambda_{q_j}$ be the DN map for the equation $\Delta_g u + q u^m = 0$ in $M$ with small Dirichlet data. If $\Lambda_{q_1}=\Lambda_{q_2}$, then $q_1=q_2$. This works even when the Neumann data is measured at a single point $x_0 \in \p M$.
\end{Corollary}

As another result, we give a simplified proof of the completeness of products of four harmonic functions on a CTA manifold $(M,g)$.

\begin{Definition}[\cite{ferreira2013calderon}] \label{def_cta}
Let $(M,g)$ be a compact manifold with smooth boundary. We say that $(M,g)$ is \emph{transversally anisotropic} (TA) if $M \subset \subset \mR \times M_0$ and $g = e \oplus g_0$, where $e$ is the Euclidean metric on $\mR$ and $(M_0,g_0)$ is a compact manifold with smooth boundary. We say that $(M,g)$ is \emph{conformally transversally anisotropic} (CTA) if $(M,c^{-1} g)$ is TA for some smooth positive function $c$.
\end{Definition}

The following result was proved for a TA manifold in \cite{LASSAS202144} when $f \in C^1(M)$ and $u_j$ solve $\Delta_g u_j = 0$, and for a CTA manifold in \cite{FEIZMOHAMMADI20204683} with an additional assumption on the invertibility of a certain weighted geodesic ray transform.  For a general CTA manifold with $f \in C^{1,1}(M)$ this result follows from \cite{KrupchykUhlmann2023_APDE}. Our proof avoids all stationary and non-stationary phase arguments in the previous proofs and works for $f \in C(M)$.

\begin{Theorem}\label{thm:density_of_4}
	Let $(M,g)$ be a compact CTA manifold with smooth boundary, and let $V_j \in C^\infty(M)$ for $1 \leq j \leq m$ where $m \geq 4$. If $q \in C(M)$  satisfies 
	\begin{align}\label{integral id to zero on CTA}
	\int_M q \s u_1 \cdots u_m \,dV = 0
	\end{align}
	for any $u_j \in C^{\infty}(M)$ solving $(\Delta_g+V_j) u_j = 0$ in $M$, then $q \equiv 0$.
\end{Theorem}

The main geometric simplification in our proof is the following. Consider the case $m=4$ and let $u_j$ be suitable CGO solutions as in \cite{LASSAS202144} that depend on a large parameter $\tau > 0$. Instead of looking at two intersecting geodesics on the transversal manifold, which produces pointwise concentration of $u_1 u_2 u_3 u_4$ at the intersection points as in \cite{LASSAS202144, KrupchykUhlmann2023_APDE}, we use a fixed geodesic $\gamma$ on the transversal manifold $M_0$. Let us consider here the case where $\gamma$ does not self-intersect for simplicity. Then the product $u_1 u_2 u_3 u_4$ concentrates on a two-dimensional manifold $\Gamma = \mR \times \gamma$.  We then use the fact that the amplitudes of the CGO solutions are holomorphic or antiholomorphic functions on $\Gamma$. This yields the limit 
\[
0 = \lim_{\tau \to \infty} \int_M f \s u_1 u_2 u_3 u_4 \,dV = \int_{\Gamma} f a_1 a_2 a_3 a_4 \,dV_{\Gamma},
\]
where $a_1, a_2$ are holomorphic functions on $\Gamma$ and $a_3, a_4$ are antiholomorphic on $\Gamma$, and $dV_\Gamma$ is a positive multiple of the Riemannian volume form on $\Gamma$. More specifically, the construction gives amplitudes $a_1 = e^{i\lambda_1 z}$, $a_3 = e^{i \lambda_2 \bar{z}}$ and $a_2 = a_4 = 1$ where $\lambda_j \in \mR$ are free parameters and $z$ is a complex coordinate on $\Gamma$. The set $A = \mathrm{span} \{ a_1 a_2 a_3 a_4 \}$ is a unital subalgebra of $C(\Gamma)$ that is closed under complex conjugation and separates points. The Stone-Weierstrass theorem then implies that $A$ is dense in $C(\Gamma)$, which implies that $q|_{\Gamma}=0$. Repeating this argument for many maximal geodesics $\gamma$ on $M_0$ implies that $q=0$ everywhere.

Theorem \ref{thm:density_of_4} allows us to relax the regularity assumptions on the unknowns in the main theorems of \cite{LASSAS202144, FEIZMOHAMMADI20204683, KrupchykUhlmann2023_APDE}:

\begin{Corollary}\label{cor:CTA}
Let $(M,g)$ be a compact CTA manifold with smooth boundary. Let $q_1, q_2 \in C(M)$, and let $m \geq 3$ be an integer. Let $\Lambda_{q_j}$ be the DN map for the equation $\Delta_g u + q u^m = 0$ in $M$ with small Dirichlet data. If $\Lambda_{q_1}=\Lambda_{q_2}$, then $q_1=q_2$.
\end{Corollary}

\begin{Remark}
One could also consider the following refinement of Theorem \ref{thm:density_of_4}. Let $(M,g)$ be a compact manifold with smooth boundary, and suppose that 
\[
\int_M q u_1 u_2 u_3 u_4 \,dV = 0
\]
whenever $\Delta_g u_j = 0$ (or $(\Delta_g+V_j)u_j = 0$) in $M$. It is likely that one could prove that $q|_{\Gamma} = 0$ whenever $\Gamma$ is a good bicharacteristic leaf for some limiting Carleman weight on $(M,g)$ in the sense of \cite[Definition 1.2]{Salo_normalforms}, if the definition of a good bicharacteristic leaf is  adapted to products of four quasimodes instead of two. We recall that on a TA manifold the good bicharacteristic leaves cover $M$ up to a set of measure zero \cite[Theorem 1.2]{Salo_normalforms}, so this would indeed generalize Theorem \ref{thm:density_of_4}.
\end{Remark}

\begin{Remark}
We note that Theorem \ref{thm:density_of_4} also works in $\R^2$. That is, for products of four solutions one can use standard CGO solutions with linear phase  also in two dimensions to get density results. This is in contrast with the case of products of two solutions, where one has to use different CGO solutions with quadratic type phases in two dimensions. 
\end{Remark}

This article is organized as follows. In Section \ref{sec_euclidean} we revisit the Euclidean case and discuss density results for products of harmonic functions by using the Stone-Weierstrass theorem. In Section \ref{sec_kaehler} we consider the case of K\"ahler manifolds and study the relation of holomorphic separability, plurisubharmonic functions and the existence of charts given by global holomorphic functions. This involves some methods from the characterization of Stein manifolds via plurisubharmonic functions based on the $L^2$ estimates of H\"ormander. Finally in Section \ref{sec:CTA_proofs} we consider the case of CTA manifolds and give the proof of Theorem \ref{thm:density_of_4}.

\section*{Acknowledgements}

The authors would like to thank Bo Berndtsson for a helpful remark related to plurisubharmonic functions. T.L. was supported by the Academy of Finland (Centre of Excellence in
Inverse Modeling and Imaging, grant numbers 284715 and 309963). The research of M.S.~is supported by Research Council Finland (CoE in Inverse Modelling and Imaging and FAME flagship, grants 353091 and 359208).

\section{The Euclidean case} \label{sec_euclidean}

In this short section we review, from the point of view of the Stone-Weierstrass theorem, some classical completeness results for products of special harmonic functions in subsets of $\mR^n$. The case of exponential harmonic functions may be found in \cite{Calderon1980} and the case of Green functions follows from \cite{Riesz1938} (see the discussion in \cite{Isakov_book_1990} or \cite{ghosh2020calderon}). We also prove completeness of products of harmonic homogeneous polynomials; we do not know if this result has appeared before in the literature.

\begin{Theorem}
Let $\Omega \subset \mR^n$, $n \geq 2$, be a bounded open set (with $n \geq 3$ in the case of $S_2$ below). Then the sets 
\begin{align*}
S_1 &= \{ u_1 u_2 \,:\, u_j = e^{\rho_j \cdot x}, \ \rho_j \in \mC^n, \ \rho_j \cdot \rho_j = 0 \}, \\
S_2 &= \{ u_1 u_2 \,:\, u_j(x) = |x-x_j|^{2-n}, \ x_j \in \mR^n \setminus \ol{\Omega} \}, \\
S_3 &= \{ u_1 u_2 \,:\, \text{$u_j$ is a harmonic homogeneous polynomial in $\mR^n$} \}
\end{align*}
are complete in $C(\ol{\Omega})$.
\end{Theorem}
\begin{proof}
For $S_1$, we fix $\xi \in \mR^n$ and choose $u_j = e^{\rho_j \cdot x}$ where $\rho_1 = \eta+i\xi$ and $\rho_2 = -\eta + i\xi$ where $\eta \in \mR^n$ satisfies $|\eta| = |\xi|$ and $\eta \cdot \xi = 0$. Then $u_1 u_2 = e^{2i x \cdot \xi}$. This shows that $\mathrm{span}(S_1)$ contains the set 
\[
A = \mathrm{span}\{ e^{ix \cdot \xi} \,:\, \xi \in \mR^n \}.
\]
The set $A$ is a unital subalgebra of $C(\ol{\Omega})$ that is closed under complex conjugation and separates points in $\ol{\Omega}$. Thus $A$ is dense in $C(\ol{\Omega})$ by the Stone-Weierstrass theorem.

To prove that $S_2$ is complete, it is enough to show that if $\mu$ is a bounded measure on $\ol{\Omega}$ satisfying 
\begin{equation} \label{mu_eq}
\int_{\ol{\Omega}} |x-y|^{4-2n} \,d\mu(y) = 0, \qquad x \in \mR^n\setminus \ol{\Omega},
\end{equation}
then $\mu \equiv 0$. Multiplying \eqref{mu_eq} by $|x|^{2n-4}$ and letting $|x| \to \infty$ we obtain 
\[
\int_{\ol{\Omega}} d\mu(y) = 0.
\]
Applying powers of $\Delta_x$ to \eqref{mu_eq} and using that $n \geq 3$ yields 
\[
\int_{\ol{\Omega}} |x-y|^{4-2n-2k} \,d\mu(y) = 0, \qquad k \geq 0, \ x \in \mR^n\setminus \ol{\Omega}.
\]
Now, applying $\p_{x_j}$ to \eqref{mu_eq}, multiplying by a suitable power of $|x|$ and letting $|x| \to \infty$ gives 
\[
\int_{\ol{\Omega}} |x-y|^{4-2n-2} y_j d\mu(y) = \int_{\ol{\Omega}} y_j d\mu(y) = 0, \qquad 1 \leq j \leq n.
\]
Repeating this for various higher order derivatives and taking limits as $|x| \to \infty$ yields 
\[
\int_{\ol{\Omega}} y^{\alpha} d\mu(y) = 0, \qquad \alpha \in \mN^n.
\]
The set $A = \mathrm{span}\{ y^{\alpha} \}$ is a unital subalgebra of $C(\ol{\Omega}, \mR)$ that separates points. Thus the real Stone-Weierstrass theorem implies that $\mu \equiv 0$.

Next we consider $S_3$. Let first $n = 2m$ be even, and write points in $\mR^n$ as $(x_1,y_1, \ldots, x_m, y_m)$. Writing $z_j = x_j + i y_j$ and $z = (z_1, \ldots, z_m)$, we can choose harmonic homogeneous polynomials $u_1 = z^{\alpha}$ and $u_2 = \bar{z}^{\beta}$ where $\alpha, \beta \in \mN^n$. Thus $\mathrm{span}(S_3)$ contains the set 
\[
A = \mathrm{span}\{ z^{\alpha} \bar{z}^{\beta} \,:\, \alpha, \beta \in \mN^n \}.
\]
This is a unital subalgebra of $C(\ol{\Omega})$ that is closed under complex conjugation and separates points. Hence $A$ is dense in $C(\ol{\Omega})$ by the Stone-Weierstrass theorem.

Finally we consider $S_3$ in the case where $n = 2m+1$ is odd. For simplicity we first show that if $f \in C_c(\Omega)$ satisfies 
\[
\int_{\Omega} f u_1 u_2 \,dx = 0
\]
for all harmonic homogeneous polynomials $u_j$, then $f \equiv 0$. We extend $f$ by zero to $\mR^{2m+1}$ and write points in $\mR^{2m+1}$ as $x = (x',t)$ where $x' \in \mR^{2m}$. Choosing $u_j = u_j(x')$, we obtain 
\[
\int_{\mR^{2m}} \left[ \int_{-\infty}^{\infty} f(x',t) \,dt \right] u_1(x') u_2(x') \,dx' = 0.
\]
Now if $u_j(x')$ are harmonic homogeneous polynomials in $\mR^{2m}$, the density argument for even dimensions above (applied in a large ball in $\mR^{2m}$) implies that 
\[
\int_{-\infty}^{\infty} f(x',t) \,dt= 0, \qquad x' \in \mR^{2m}.
\]
This means that the integrals of $f$ along all lines in direction $e_{2m+1}$ vanish. There is nothing special about the direction $e_{2m+1}$, and repeating this argument for other directions implies that the integrals of $f$ over all lines in $\mR^{2m+1}$ must be zero. By injectivity of the X-ray transform \cite{helgason1999radon} we see that $f \equiv 0$ as required. To show that $S_3$ is complete in $C(\ol{\Omega})$ it is enough to apply the argument above with $f$ replaced by a bounded measure $\mu$ on $\ol{\Omega}$ and to use injectivity of the X-ray transform on compactly supported distributions \cite{StefanovUhlmann_book}.
\end{proof}

\section{K\"ahler manifolds} \label{sec_kaehler}

In the introduction we already proved that products of harmonic functions are complete on compact holomorphically separable K\"ahler manifolds with boundary. Here we first give a simple example of complex manifolds that are not holomorphically separable.

\begin{Example}
Let $M$ be a compact complex manifold with $C^{\infty}$ boundary, and suppose that $S$ is a closed (i.e.\ compact without boundary) connected embedded complex submanifold of $M^{\mathrm{int}}$. Then for any points $x, y \in S$ with $x \neq y$ and for any holomorphic function $f$ in $M^{\mathrm{int}}$, the function $f|_S$ is holomorphic in $S$ and hence $f|_S$ is constant in $S$. Thus in such a case holomorphic functions never separate points of $S$.

An explicit example is obtained from the complex projective space $\mC P^n = (\mC^{n+1} \setminus \{0\})/\!\sim$ where $(z_0, \ldots, z_n) \sim (w_0, \ldots, w_n)$ if $(w_0,\ldots,w_n) = \lambda(z_0,\ldots,z_n)$ for some $\lambda \in \mC \setminus \{0\}$. Let $S = \{ [z_0, \ldots, z_n] \,:\, z_n = 0 \}$ be the equator, and let $M = f^{-1}([0, \eps])$ where $\eps > 0$ is small and 
\[
f([z_0,\ldots,z_n]) = \frac{|z_n|^2}{|z|^2}.
\]
Then $M$ is a compact subdomain of $\mC P^n$ with smooth boundary such that $S$ is an embedded complex submanifold of $M^{\mathrm{int}}$. Thus holomorphic functions do not separate points in $M$.
\end{Example}

Next we will show that holomorphic separability (for functions holomorphic near $M$) is equivalent to the existence of a smooth strictly plurisubharmonic function. This is very similar to the argument that a manifold is Stein if and only if it admits a strictly plurisubharmonic function (solution of the Levi problem, see e.g.\ \cite[Theorem 5.2.10]{Hormander_complex}). However, we need to verify that the argument works also for compact manifolds with boundary.

\begin{Definition}
Let $X$ be an open complex manifold. A function $\varphi \in C^{2}(X)$ is plurisubharmonic if for any $x \in X$, the Levi matrix (computed in some complex coordinates at $x$) 
\[
H_{\varphi}(x) =  \left( \frac{\p^2 \varphi(x)}{\p {z_j} \p {\bar{z}_k}} \right)_{j,k=1}^n
\]
is positive semidefinite. We say that $\varphi$ is strictly plurisubharmonic if this matrix is positive definite at each $x \in X$.

More generally, a function $\varphi \in L^1_{\mathrm{loc}}(X)$ is plurisubharmonic if for any $a \in \mC^n$, the distribution  
\[
\sum \frac{\p^2 \varphi}{\p {z_j} \p {\bar{z}_k}} a_j \bar{a}_k
\]
is a nonnegative measure on $X$.
\end{Definition}

If $\dim_{\mC}(M) = 1$, then plurisubharmonic functions are precisely the subharmonic functions, and in general any plurisubharmonic function is subharmonic \cite[{\S}I.5]{Demailly}. We also note that if $f$ is holomorphic, then $|f|^2$ is plurisubharmonic.

We will prove the following equivalence. Below we assume that $M$ is a compact subdomain in an open complex manifold $X$, and we say that a property holds near $M$ if it holds in some open set in $X$ containing $M$.

\begin{Theorem} \label{lem_sep}
Let $M$ be a compact K\"ahler manifold with $C^{\infty}$ boundary. The following are equivalent:
\begin{enumerate}
\item[(a)] 
For any $x, y \in M$ with $x \neq y$, there is a holomorphic function $f$ near $M$ with $f(x) \neq f(y)$. 
\item[(b)]
There is a strictly plurisubharmonic function $\varphi$ near $M$.
\end{enumerate}
\end{Theorem}

One direction is easy and we prove it  following the argument in \cite[Lemma 6.17 in {\S}I]{Demailly}.

\begin{proof}[Proof of first implication in Lemma \ref{lem_sep}]
Assume that (a) holds. Fix $x_0 \in M$. Choose complex coordinates $z$ in a neighborhood $V$ of $x_0$  in $M$ such that $x_0$ corresponds to $0$ and (after scaling if necessary) $V$ contains $\{ |z| \leq 1 \} \cap M$. By (a), for any $y \in \{ |z| = 1 \} \cap M$ there is $f_y$ holomorphic near $M$ with $|f_y(y)| = 2$ and $f_y(x_0) = 0$. Then by compactness there are finitely many functions $f_1, \ldots, f_N$ that are holomorphic near $M$ such that $v_{x_0} := \sum |f_j|^2$ satisfies $v_{x_0}(x_0) = 0$ and $v_{x_0}(y) \geq 2$ for $y \in \{ |z|=1 \} \cap M$. Note that $v_{x_0}$ is $C^{\infty}$, nonnegative and plurisubharmonic in some neighborhood $U$ of $M$. By continuity, one can further choose $U$ so that $v_{x_0}(y) \geq 1$ whenever $y \in \{ |z|=1 \} \cap U$.

Next we define 
\[
u_{x_0}(z) = \left\{ \begin{array}{ll} v_{x_0}(z) & \text{in $U \setminus \{ |z| <1 \}$}, \\[3pt] \max \{ v_{x_0}(z), (|z|^2+1)/3 \} & \text{in $\{ |z| \leq 1\} \cap U$}. \end{array} \right.
\]
Then $u_{x_0} = v_{x_0}$ near $\{ |z|=1 \} \cap U$ and $u_{x_0} = (|z|^2+1)/3$ near $x_0$. The function $u_{x_0}$ is not smooth everywhere, but it is continuous, nonnegative and plurisubharmonic in $U$ \cite[{\S}I.5]{Demailly} and strictly plurisubharmonic near $x_0$. By Richberg's approximation theorem \cite[{\S}I.5]{Demailly} there is a plurisubharmonic function $\tilde{u}_{x_0}$ that is $C^{\infty}$, nonnegative and plurisubharmonic in $U$ and strictly plurisubharmonic in some neighborhood $U_{x_0}$ of  $x_0$.  After covering $M$ by finitely many such neighborhoods $U_{x_0}$, we obtain functions $\tilde{u}_1, \ldots, \tilde{u}_m$ such that $\varphi := \tilde{u}_1 + \cdots + \tilde{u}_m$ is $C^{\infty}$ and strictly plurisubharmonic everywhere near $M$. Thus (b) holds.
\end{proof}

We now move to the implication (b) $\implies$ (a). This is a consequence of an interpolation theorem.

\begin{Theorem} \label{thm_interpolation_scv}
Let $M$ be a compact subset of an open K\"ahler manifold $X$ and suppose that there is a strictly plurisubharmonic $C^{\infty}$ function near $M$. Given $m \geq 0$ and a finite set of points $x_{\nu}$ in $M$, one can find a holomorphic function $f$ near $M$ with prescribed Taylor expansions to order $m$ at each $x_{\nu}$.
\end{Theorem}

We will give a proof based on plurisubharmonic weights having logarithmic singularities. See \cite[Theorem 5.3]{Ohsawa} or \cite[Theorem 9.18]{Demailly_notes} for corresponding results on open pseudoconvex manifolds. The proof involves a scheme for approximating singular plurisubharmonic functions by smooth ones \cite{Demailly1982_approximation} (this is the only point where we need that the manifold is K\"ahler). It is convenient to state plurisubharmonicity in more invariant terms. If $d = \partial + \dbar$ where $\partial$ and $\dbar$ are the Dolbeault operators, one has 
\[
i \partial \dbar \varphi = i \frac{\p^2 \varphi(x)}{\p {z_j} \p {\bar{z}_k}} \,dz^j \wedge d\bar{z}^k.
\]
A $(1,1)$-form $u = i u_{jk} \,dz^j \wedge d\bar{z}^k$ is said to be positive semidefinite if the matrix $(u_{jk})$ is positive semidefinite. In this case we write $u\geq 0$. Then $\varphi$ is plurisubharmonic if $i \partial \dbar \varphi \geq 0$. Moreover, if $X$ is a K\"ahler manifold with fundamental $(1,1)$-form $\omega$ written in local coordinates as $\omega = \frac{i}{2} h_{jk} \,dz^j \wedge d\bar{z}^k$ where $h = h_{jk} \,dz^j \otimes dz^k$ is the corresponding Hermitian metric, the condition of strict plurisubharmonicity may be written as $i \partial \dbar \varphi \geq c \omega$ for some $c > 0$.

\begin{proof}[Proof of Theorem \ref{thm_interpolation_scv}]
We follow the argument in \cite[Theorem 9.18]{Demailly_notes}. Let $\{ x_{\nu} \}_{\nu=1}^N$ be a finite subset of $M$ and let $P_{\nu}(z^{\nu})$ be polynomials of degree $m_{\nu}$ where $z^{\nu} = (z^{\nu}_1, \ldots, z^{\nu}_n)$ is a complex coordinate chart in a small neighborhood $U_{\nu}$ of $x_{\nu}$. We want to find a holomorphic function $f$ near $M$ such that the Taylor expansion of $f$ to order $m_{\nu}$ at $x_{\nu}$ agrees with $P_{\nu}$.
Let $\theta_{\nu} \in C^{\infty}_c(U_{\nu})$ be a cutoff function with $\theta_{\nu} = 1$ near $x_{\nu}$ and $0 \leq \theta_{\nu} \leq 1$. The point is to find a holomorphic function $f$ in the form $f = \sum \theta_{\nu} P_{\nu} + u$, where $u$ solves 
\begin{equation} \label{dbaru_v_eq}
\dbar u = v := -\dbar(\sum \theta_{\nu} P_{\nu}), \qquad \text{and $u$ vanishes to order $m_{\nu}$ at $x_{\nu}$}.
\end{equation}

We first define a function 
\[
\varphi_0 = \sum_{\nu} 2(n+m_{\nu}) \theta_{\nu} \log\,|z^{\nu}|.
\]
The function $\varphi_0$ is in $L^1_{\mathrm{loc}}(X)$, it has a logarithmic singularity at the points $x_{\nu}$ and it is smooth elsewhere. The complex Hessian $H_{\varphi_0}$ is a nonnegative measure in the sets where $\theta_{\nu} = 1$, and the negative part of $H_{\varphi_0}$ is bounded from below away from these neighborhoods. Next we let $\Omega \subset X$ be an open set such that $M \subset \Omega$ and $\psi$ is strictly plurisubharmonic near $\ol{\Omega}$, and consider 
\[
\tilde{\varphi} = \chi_0(\psi) + \varphi_0,
\]
where $\chi_0$ is a strictly increasing convex function. Such $\Omega$ and $\psi$ exist by assumption. By adding a constant we may also assume $\psi \geq 0$. Since 
\[
i\partial \dbar (\chi_0(\psi)) = i \chi_0''(\psi) \partial \psi \wedge \dbar \psi +  i \chi_0'(\psi) \partial \dbar \psi,
\]
we can choose $\chi_0$ e.g.\ as $\chi_0(t) = e^{\lambda t}$ for $\lambda > 0$ large so that $\tilde{\varphi}$ becomes plurisubharmonic near $\ol{\Omega}$. By \cite[Theorem 0.7]{Demailly1982_approximation}, there is a nonincreasing sequence $(\tilde{\varphi}_{j})$ of $C^{\infty}$ functions near $\ol{\Omega}$ with $\tilde{\varphi}_{j} \to \tilde{\varphi}$ pointwise on $\ol{\Omega}$ and $i \partial \dbar \tilde{\varphi}_{j} \geq - \omega$ on $\ol{\Omega}$. Finally fix $R > 0$ large and choose 
\[
\varphi_{j} = \chi_1(\psi) + \tilde{\varphi}_{j}, \qquad \varphi =  \chi_1(\psi) + \tilde{\varphi},
\]
where $\chi_1$ is another convex function chosen so that $i \partial \dbar \varphi_{j} \geq R \omega$ for all $j$.

Define the norm 
\[
\norm{f}_{\varphi}^2 = \int_{\Omega} |f|^2 e^{-\varphi} \,dV.
\]
The functions $\varphi_j$ are smooth and strictly plurisubharmonic near $\ol{\Omega}$. If $R > 0$ was chosen large enough to begin with, we can apply H\"ormander's $L^2$-estimate and the related solvability result \cite[Theorem 5.2.4 and Corollary 5.2.6]{Hormander_complex} together with the condition $\dbar v = 0$ to obtain a $C^{\infty}$ solution $u_j$ of 
\[
\dbar u_j = v \text{ in $\Omega$}
\]
satisfying 
\[
\norm{u_j}_{\varphi_j} \leq \norm{v}_{\varphi_j}.
\]
Whenever $j \geq j_0$ we have $\varphi_{j_0} \geq \varphi_j \geq \varphi$, which gives 
\[
\norm{u_j}_{\varphi_{j_0}} \leq \norm{u_j}_{\varphi_j} \leq \norm{v}_{\varphi}.
\]
Note that, writing $\chi = \chi_0 + \chi_1$,  
\[
\norm{v}_{\varphi}^2 = \sum \int_{U_{\nu}} |P_{\nu}|^2 |\dbar \theta_{\nu}|^2 e^{-2(n+m_{\nu}) \theta_{\nu} \log\,|z^{\nu}| - \chi(\psi)} \,dV.
\]
Since $\dbar \theta_{\nu} = 0$ near the singular points $x_{\nu}$, it follows that $\norm{v}_{\varphi}$ is finite. Thus for any $j_0$ there is a subsequence of $(u_j)$ converging weakly in the $\norm{\,\cdot\,}_{\varphi_{j_0}}$ norm. A diagonal argument gives $u$ solving $\dbar u = v$ with $\norm{u}_{\varphi_{j_0}} \leq \norm{v}_{\varphi}$ for any $j_0$. Then monotone convergence gives 
\[
\norm{u}_{\varphi} \leq \norm{v}_{\varphi}.
\]
By looking at the norm on the left, we see that $u$ must satisfy 
\[
\int_{U_{\nu}} \frac{|u|^2}{|z^{\nu}|^{2(n+m_{\nu})}} < \infty.
\]
This means that $u$ vanishes to order $m_{\nu}$ at $x_{\nu}$, proving \eqref{dbaru_v_eq}.
\end{proof}

To conclude this section we show that holomorphic separability implies the existence of local charts given by global holomorphic functions, thus proving that condition (b) in \cite{guillarmou2019linearized} was not really necessary (at least if we consider holomorphic functions near $M$).

\begin{Theorem}
Let $M$ be a compact K\"ahler manifold with $C^{\infty}$ boundary, and suppose that for any $x, y \in M$ with $x \neq y$ there is a holomorphic function $f$ near $M$ with $f(x) \neq f(y)$. Then for any $x \in M$ there are holomorphic functions $f_1, \ldots, f_n$ near $M$ which form a coordinate system near $x$.
\end{Theorem}
\begin{proof}
By Theorem \ref{lem_sep} there is a strictly plurisubharmonic function near $M$. Fix $x \in M$ and choose some complex coordinates $z$ near $x$. Theorem \ref{thm_interpolation_scv} ensures that when $1 \leq k \leq n$ there is a holomorphic function $f_k$ near $M$ with 
\[
f_k(x) = 0, \qquad \p_{z_j} f_k(x) = \delta_{jk}.
\]
The functions $f_1, \ldots, f_n$ have the required property.
\end{proof}

\section{CTA manifolds}\label{sec:CTA_proofs}

In this section we prove Theorem \ref{thm:density_of_4} and Corollary \ref{cor:CTA}. Before that, let us recall the construction of CGO solutions on CTA manifolds.
These are solutions to the equation
\begin{equation}\label{eq:shrode}
 (\Delta_g+V) v = 0
\end{equation}
on a manifold $M$ compactly contained in $I\times M_0$ equipped with metric $g=c(e\oplus g_0)$. Here $M_0$ is a manifold with boundary, $I\subset \R$ is a closed interval and $c>0$  and $V$ are real valued $C^\infty$ functions on $M$. These solutions originate from the works \cite{DosSantosFerreira2009} and \cite{ferreira2013calderon}. Higher order Sobolev estimates  for the related correction terms have been obtained in \cite{FEIZMOHAMMADI20204683, LASSAS202144, KrupchykUhlmann2023_APDE}.

If we replace $v$ above by $c^{-\frac{n-2}{4}} v$, it is enough to construct solutions for the TA metric $g = e \oplus g_0$ with a new potential $V$ (see e.g.\ \cite[Section 2]{KrupchykUhlmann2023_APDE}). We will assume below that this reduction has been done.

\begin{Proposition}[CGO solutions \cite{KrupchykUhlmann2023_APDE}] \label{prop:CGO_construction}
 Let $(M,g)$ be a TA manifold with smooth boundary $\p M$, $\dim(M)=n\geq 3$, and $V\in C^\infty(M)$. Let $\gamma:[0,T]\to M_0$ be a nontangential geodesic, and let $\lambda\in \C$. For any $K\in \N$ and $k\in \N$, there is a family of functions $u=u_s\in C^\infty(M)$, where $s=\tau+i\lambda \in \C$ with $\tau \in \mR$ and $|\tau|$ large, such that  
 \begin{equation}\label{Hk_and_L4}
\begin{split}
  (\Delta_g+V)u_s=0, \\
  u_s=\tau^{\frac{n-2}{8}}e^{\pm sx_1}(v_s+r_s),
\end{split}
\end{equation}
where $\norm{r_s}_{H^k(M)}=O(\tau^{-K})$ as $\tau \to \infty$.

 The functions $v_s$ have the following properties.
 If $p\in \gamma([0,T])$, then there is $P\in \N$ such that on a neighborhood $U$ of $p$ the function $v_{s}$ is a finite sum 
	\begin{equation}\label{finite_sum}
	v_s= v^{(1)} + \cdots + v^{(P)}
	\end{equation}
	on $I\times U$, where $t_1 < \cdots < t_P$ are the times in $[0,T]$ such that $\gamma(t_l) = p$, $l=1,\ldots,P$. Each $v^{(l)}$ has the form 
	\begin{equation}\label{form_of_gb}
	v^{(l)} = e^{is \psi^{(l)}} a^{(l)},
	\end{equation}
	where each $\psi = \psi^{(l)}$ is a smooth complex function in $U$ satisfying   
	\begin{align}\label{Phi_prop}
	\begin{split}
	&\psi(\gamma(t)) = t, \quad \nabla \psi(\gamma(t)) = \dot{\gamma}(t), \\ & \mathrm{Im}(\nabla^2 \psi(\gamma(t))) \geq 0, \quad \mathrm{Im}(\nabla^2 \psi)(\gamma(t))|_{\dot{\gamma}(t)^{\perp}} > 0,
	\end{split}
	\end{align}
	for $t$ close to $t_l$.
	Here $a^{(l)} = a_0^{(l)} + O_{L^\infty}(\tau^{-1})$, where $a_0^{(l)}$ is independent of $x_1$ and $\tau$, $a_0^{(l)}(\gamma(t))$ is nonvanishing, and the support of $a^{(l)}$ can be taken to be in any neighborhood of $I\times \gamma([0,T])$ chosen beforehand.
\end{Proposition}

The functions $v_s$ in the proposition above are called quasimodes. We use CGOs to prove Theorem \ref{thm:density_of_4}. 
\begin{proof}[Proof of Theorem \ref{thm:density_of_4}]
\textbf{Step 1.\ The choice of CGOs:} 
Let us extend $(M,g)$ smoothly to a manifold $\widetilde M:=I\times  M_0$ equipped with the metric $e\oplus g_0$. Here $I \subset \R$ is a closed interval. We also extend $q$ by zero to $\widetilde M$ as an $L^\infty(\widetilde M)$ function. From the assumption \eqref{integral id to zero on CTA}, it then follows that
\begin{equation}\label{eq_density_of_4}
\int_{\widetilde M} q\s u_1 \cdots u_m \,dV = 0, 
\end{equation}
where each $u_k$, $k=1,\ldots,m$, solves $(\Delta_g+V_k) u_k = 0$ in $\M$. 
We first assume that $m=4$, and then consider the case $m>4$ separately.

Assume that $\gamma: [0,T]\to M_0$ is a unit speed nontangential geodesic in $M_0$. We choose complex geometrics optics solutions of the form \eqref{Hk_and_L4} as the solutions $u_k$, $k=1,\ldots,4$. We specifically choose them to be of the form
 \begin{equation}\label{eq:choice_of_CGOs}
  \begin{split}
 u_1&=\tau^{\frac{n-2}{8}}e^{(\tau +i\lambda_1)x_1}(v^1_{\tau+i\lambda_1}+r_1) \\
 u_2&=\tau^{\frac{n-2}{8}}e^{-\tau x_1}(v^2_{\tau}+r_2) \\
 u_3&=\tau^{\frac{n-2}{8}}\overline{e^{(\tau+i\lambda_2) x_1}(v^3_{\tau+i\lambda_2} +r_3)} \\
 u_4&=\tau^{\frac{n-2}{8}}\overline{e^{-\tau x_1}(v^4_{\tau}+r_4)},
 \end{split}
 \end{equation}
 where $\norm{r_{k}}_{H^K(\M)} \lesssim \tau^{-N}$  and $n=\dim(\M)=\dim(M_0)+1$. We will choose $K$ and $N$ large beforehand. 
 The first and the third of these CGOs correspond to choosing $s=\tau+i\lambda_1$ and $s=\tau+i\lambda_2$ in 
 \eqref{Hk_and_L4} respectively. Substitution of these CGOs $u_k$ into \eqref{eq_density_of_4} gives
\begin{equation}\label{eq:subs_cgos2}
  0=\tau^{\frac{n-2}{2}}\int_{\M} q e^{i\lambda_1x_1}e^{-i\lambda_2x_1}v^1_{\tau+i\lambda_1}v^2_{\tau}\overline{v^3_{\tau+i\lambda_2}v^4_{\tau}}\s dV+ R_\tau,
 \end{equation}
 where $R_\tau$ corresponds to the contributions from the correction terms $r_k$.  We take $\norm{r_{k}}_{H^K(M)} \lesssim \tau^{-N}$ for $K$ and $N$ large, so that $\norm{r_{k}}_{L^{\infty}(M)} \lesssim \tau^{-N}$ by Sobolev embedding. Thus 
 $R_\tau$ satisfies
 \[
  R_\tau=O(\tau^{-1}).
 \]
 Here we also used the fact that $q\in L^\infty(\M)$. As the integral in \eqref{eq:subs_cgos2} will be of size $1$ as $\tau\to \infty$, we may neglect $R_\tau$ in the following analysis.
 
 In the case where $\gamma$ does not self-intersect, the quasimodes $v^1_{\tau+i\lambda_1},\ldots, v^4_{\tau}$ on $\M$ are given by 
 \begin{equation}\label{eq:no_self_intersections_vs}
  \begin{split}
v^1_{\tau+i\lambda_1} &= e^{i(\tau+i\lambda_1) \psi} a^{(1)}, \\
v^2_{\tau} &= e^{i\tau \psi} a^{(2)}, \\
v^3_{\tau+i\lambda_2} &= e^{i(\tau+i\lambda_2) \psi} a^{(3)}, \\
v^4_{\tau} &= e^{i \tau \psi} a^{(4)}
\end{split}
\end{equation}
by Proposition~\ref{prop:CGO_construction}. Moreover, we have 
\[
 a^{(k)}(x_1,x')=a_{0}(x')+O_{L^\infty}(\tau^{-1}), \quad k=1,\ldots,4.
\]
We consider the case where $\gamma$ does not self-intersect separately in Step 2 below to convey the idea of the proof better. A reader interested only in this case can jump directly to Step 2 from here.
 
 In general, the geodesic $\gamma$ can have self-intersections. In this case the quasimodes $v^1_{\tau+i\lambda_1},\ldots, v^4_{\tau}$ have the following properties. Let $p\in \gamma([0,T])$. By Proposition~\ref{prop:CGO_construction} the quasimode $v^1_{\tau+i\lambda_1}$ is a finite sum 
	\begin{equation}\label{eq:gaussian_beam_finite_sum}
	v^1_{\tau+i\lambda_1}|_U = v^{(1,1)} + \cdots + v^{(1,P)}
	\end{equation}
	on a small enough neighborhood $U$ of $p$, where 
	\begin{equation}\label{eq:intersection_times}
	  t_1< \cdots < t_P
	\end{equation}
are the times in $[0,T]$ such that $\gamma(t_j) = p$ and $P$ depends on $p$. We choose $U$ so small that the geodesic $\gamma$ self-intersects only at $p$ in $U$, or not at all in $U$.  Moreover, there are intervals $I_j$, $j=1,\ldots,P$, such that $t_j\in I_j$ and
 \begin{equation}\label{eq:interval_conds}
  \text{supp}(v^{(1,j)})\cap \gamma(I_j)=\gamma(I_j), \quad \gamma(I_{j})\cap \gamma(I_{j'})= \{ p \}, \quad 
  I_j\cap I_{j'}=\emptyset \text{ for } j\neq j'
 \end{equation}
 holding for all $|\tau|$ large. 
 Each $v^{(1,j)}$ has the form 
	\begin{equation}\label{eq:vj_form}
	 v^{(1,j)} = e^{i(\tau+i\lambda_1) \psi^{(j)}} a^{(1,j)},
	\end{equation}
	where each $\psi = \psi^{(j)}$ is a smooth complex function defined in $U$ satisfying 
	\begin{align}\label{eq:phase_properties}
		\psi(\gamma(t)) = t, \quad \nabla \psi(\gamma(t)) = \dot{\gamma}(t), \quad \mathrm{Im}(\nabla^2 \psi(\gamma(t))) \geq 0, \quad \mathrm{Im}(\nabla^2 \psi(\gamma(t)))|_{\dot{\gamma}(t)^{\perp}} > 0,
	\end{align}
	 for $t\in I_j$. Each $a^{(1,j)}$ can be taken to be  supported in any fixed neighborhood of $I\times \gamma(I_j)$ in $\M$. 
	 We also have that 
 \begin{equation}\label{eq:amplitude}
  a^{(1,j)}=a_{0}^{(j)}+s^{-1}a_{1}^{(1,j)}+s^{-2}a_{2}^{(1,j)}+\cdots,
 \end{equation}
 with $s=\tau +i\lambda_1$, and
  \[
 a_0^{(j)}|_{I\times \gamma(I_j)}> 0.
 \]
  
 We have representations similar to \eqref{eq:gaussian_beam_finite_sum} for $v^2_{\tau}$, $v^3_{\tau+i\lambda_2}$ and $v^4_{\tau}$ as a sum of $P$ functions $v^{(k,j)}$ of the form \eqref{eq:vj_form}, $k=1,\ldots,4$, $j=1,\ldots, P$. Especially, the phase functions $\psi^{(j)}$ and leading order coefficients of the corresponding amplitudes $a_{0}^{(j)}$ are the same for all $k=1,\ldots,4$. Also the corresponding intervals $I_j$ are the same for all $k=1,\ldots,4$.

\textbf{Step 2.\ The case of no self-intersections:}
Let us first consider the special case where $\gamma$ does not have self-intersections. In this case the quasimodes $v^1_{\tau+i\lambda_1},\ldots, v^4_{\tau}$ were given by \eqref{eq:no_self_intersections_vs}.
It follows that the integral identity \eqref{eq:subs_cgos2} reads
\begin{equation}\label{eq:integ_id_no_self_intersects}
 0=\tau^{\frac{n-2}{2}}\int_{\M} qe^{-4\tau \text{Im}(\psi)} e^{i\lambda_1(x_1+i\psi)} \ol{e^{i\lambda_2(x_1+i\psi)}} a^{(1)} a^{(2)} \bar{a}^{(3)} \bar{a}^{(4)} \,dV+ R_\tau,
\end{equation}
where $R_\tau=O(\tau^{-1})$ and $a^{(k)}=a_{0}+O_{L^\infty}(\tau^{-1})$.

Let us then compute the limit $\tau\to \infty$ of \eqref{eq:integ_id_no_self_intersects} using Fermi coordinates $(t,y)\in \R\times \R^{n-2}$ (see e.g.\ \cite[Lemma 3.5]{ferreira2013calderon}). By the properties \eqref{Phi_prop}, we have that
\begin{equation} \label{imphi_taylor_series_no_self}
\mathrm{Im}(\psi)(t,y) = \frac{1}{2}\mathrm{Hess}_y(\psi)(t,0) y \cdot y + O(\abs{y}^3),
\end{equation}
where $\mathrm{Hess}_y(\psi)(t,0)$ is the Hessian in the $y$-directions. In Fermi coordinates, the amplitudes $a^{(k)}=a^{(k)}(x_1,t,y)$ are supported in $I\times [0,T]\times B$, where $B$ is a neighborhood of the origin in the $y$-variables. By writing 
\[
f= q e^{i\lambda_1(x_1+i\psi)} \ol{e^{i\lambda_2(x_1+i\psi)}},
\]
we have that the integral in  \eqref{eq:integ_id_no_self_intersects} equals
\begin{multline}\label{eq:limit_identity}
 \tau^{\frac{n-2}{2}}\int_{I\times M_0} f e^{-4\tau \text{Im}(\psi)}a^{(1)}a^{(2)}\overline{a}^{(3)}\overline{a}^{(4)} \,dV \\
 = \tau^{\frac{n-2}{2}} \int_{I\times [0,T]\times B} f(x_1,t,y) e^{-4\tau \text{Im}(\psi)(t,y)}\big(\abs{a_0(x_1,t,y)}^4+O_{L^\infty}(\tau^{-1})\big) |g(x_1,t,y)|^{1/2} \,dt \,dy \,dx_1 
\\ 
= \tau^{\frac{n-2}{2}} \tau^{-\frac{n-2}{2}} \int_{I\times [0,T]\times\R^{n-2}} f(x_1,t,y/\tau^{1/2}) e^{-4\tau \text{Im}(\psi)(t,y/\tau^{1/2})}\big(\abs{a_0(x_1,t,y/\tau^{1/2})}^4+O_{L^\infty}(\tau^{-1})\big) \\
\times |g(x_1,t,y/\tau^{1/2})|^{1/2} \,dt \,dy \,dx_1.
\end{multline}
The limit $\tau\to\infty$ of the above is
\begin{multline}\label{eq:limit_identity_final}
\int_{I\times [0,T]\times\R^{n-2}} f(x_1,t,0) e^{-4\text{Im}(\text{Hess}(\psi))(t,y)}\abs{a_0(x_1,t,0)}^4 |g(x_1,t,0)|^{1/2} \,dt \,dy \,dx_1 \\
=\int_{I\times [0,T]} c(t)f(x_1,t,0)\abs{a_0(x_1,t,0)}^4 |g(x_1,t,0)|^{1/2} \,dt \,dx_1, 
\end{multline}
where   
\begin{equation*}\label{eq:ct}
c(t)=\int_{\R^{n-2}}  e^{-4\text{Im}(\text{Hess}(\psi))(t,y)} \,dy \neq 0.
\end{equation*}
Here we have used \eqref{imphi_taylor_series_no_self} and the Lebesgue dominated convergence theorem. The latter was justified by the conditions \eqref{eq:phase_properties}, which especially imply that 
\[
 \mathrm{Im}(\text{Hess}(\psi))(t,0)>0, \quad t\in [0,T].
\]
We also have $|g(x_1,t,0)|=1$ due to a property of Fermi coordinates, and (see e.g. \cite[Eqs. 56 and 62]{FEIZMOHAMMADI20204683})
\[
 a_{0}(x_1,t,0)>0.
\]

\textbf{Step 3.\ Stone-Weierstrass theorem:}
Recall that $f=q e^{i\lambda_1x_1}e^{-\lambda_1\psi}e^{-i\lambda_2x_1}e^{-\lambda_2\psi}$. Then, using the first property in \eqref{eq:phase_properties} and taking the limit $\tau\to\infty$ of \eqref{eq:integ_id_no_self_intersects} yields
\begin{equation}\label{eq:final_limit}
 0=\int_{I\times [0,T]} c(t)e^{i\lambda_1(x_1+it)} \ol{e^{i\lambda_2(x_1+it)}}q(x_1,\gamma(t))\abs{a_0(x_1,\gamma(t))}^4 \,dx_1 \,dt.
\end{equation}
Here we used the limit \eqref{eq:limit_identity_final} of \eqref{eq:limit_identity}. Let us denote 
 \[
  \tilde q(x_1,t)=c(t)q(x_1,\gamma(t))\abs{a_0(x_1,\gamma(t))}^4.
 \]
 We may take $\lambda_1, \lambda_2 \in \mR$ and differentiate \eqref{eq:final_limit} any number of times in $\lambda_1$ and $\lambda_2$ at $\lambda_1=\lambda_2=0$. This gives
\begin{equation}\label{eq:final_identity_z}
 0=\int_{I\times [0,T]} (x_1+it)^a(x_1-it)^b\tilde q(x_1,t) \,dt \,dx_1
\end{equation}
for any $a, b \geq 0$. By the complex Stone-Weierstrass theorem, the algebra generated by powers of $z = x_1+it$ and $\overline{z} = x_1-it$ is dense in $C(I\times [0,T])$ where $I \times [0,T] \subset \C$. It follows that 
\[
 \tilde{q} \equiv 0.
\]
Since $c(t)\neq 0$ for all $t\in[0,T]$  and $a_0|_{I\times [0,T]}> 0$, it follows that
\[
 q|_{I\times \gamma([0,T])}=0.
\]

To conclude the proof (in the case where $\gamma$ does not have self-intersections), we use some well-known arguments. By \cite[Lemma 3.1]{Salo_normalforms}, there is a dense set of points $D\subset M_0$ such that for all $x\in D$ there is a non-tangential unit speed geodesic $\gamma_x$ that passes through $x$. Thus, by varying the non-tangential geodesic $\gamma$ in the above argument we obtain that $q=0$ on a dense set. Thus $q\equiv 0$ by continuity. If $m>4$, we may choose the first four solutions as the solutions $u_1,\ldots,u_4$ above, and the remaining $m-4$ solutions independent of $\tau$. Using the above argument then shows that $q\s u_5\cdots u_m\equiv 0$. By Runge approximation (see e.g.\ \cite{lassas2018poisson}), for each $x_0\in \M$ the solutions $u_5,\ldots, u_m$ can be chosen so that $u_5(x_0)\cdots u_m(x_0)\neq 0$. By letting $x_0$ range over all points of $\M$, it follows that $q\equiv 0$ also in the case $m>4$. 

\textbf{Step 4.\ Analysis of self-intersections:} Let us finally consider the case where $\gamma$ can have self-intersections. Let $\chi$ be a smooth cutoff function supported in the neighborhood $U$ of $p\in \gamma([0,T])$. According to \eqref{eq:gaussian_beam_finite_sum}, in $I\times U$ we have 
\begin{multline}\label{eq:product_expansion}
 v^1_{\tau+i\lambda_1}v^2_{\tau}\overline{v^3_{\tau+i\lambda_2}v^4_{\tau}}=\sum_{j=1}^P v^{(1,j)}v^{(2,j)}\overline{v^{(3,j)}v^{(4,j)}}\\
 +\sum_{j_1,j_2,j_3,j_4=1, \ j_a\neq j_b \text{ for some } a,b=1,\ldots,4 }^P v^{(1,j_1)}v^{(2,j_2)}\overline{v^{(3,j_3)}v^{(4,j_4)}}.
\end{multline}
Let us consider the integral
\[
 \tau^{\frac{n-2}{2}}\int_{\M} \chi\s  q\s  e^{i\lambda_1x_1}e^{-i\lambda_2x_1} v^{(1,j)}v^{(2,j)}\overline{v^{(3,j)}v^{(4,j)}} \,dV
\]
which corresponds to the first sum of \eqref{eq:product_expansion}. Integrals of this type will correspond to the leading order part of our integral identity \eqref{eq:subs_cgos2}. Using the formula \eqref{eq:vj_form} for $v^{(k,j)}$, $k=1,\ldots,4$, the above integral reads 
\begin{equation}\label{eq:single_interval}
 \tau^{\frac{n-2}{2}}\int_{\M} \chi q e^{-4\tau \text{Im}(\psi)} e^{i\lambda_1(x_1+i\psi)} \ol{e^{i\lambda_2(x_1+i\psi)}}a^{(1)}a^{(2)}\overline{a}^{(3)}\overline{a}^{(4)} \,dV.
\end{equation}
Here, for simplicity of notation, we have denoted $\psi=\psi^{(j)}$ and $a^{(k)}=a^{(k,j)}$ for $j=1,\ldots,P$ and $k=1,\ldots,4$. 

By writing 
\[
f=\chi\s q\s e^{i\lambda_1(x_1+i\psi)}e^{\ol{i\lambda_2(x_1+i\psi)}},
\]
the same computation as in \eqref{eq:limit_identity} and \eqref{eq:limit_identity_final} yields that the limit $\tau\to\infty$ of \eqref{eq:single_interval} is
\begin{equation}\label{eq:limit_identity2}
\int_{I\times I_j} c(t)f(x_1,t,0)\abs{a(x_1,t,0)}^4 |g(x_1,t,0)|^{1/2} \,dt \,dx_1. 
\end{equation}
Here   
\[
c(t)=\int_{\R^{n-2}}  e^{-4\text{Im}(\text{Hess}(\psi))(t,y)} \,dy \neq 0
\]
defines a smooth function on $[0,T]$.  
Thus by the properties \eqref{eq:interval_conds} 
 of the intervals $I_j$ we see that 
\begin{multline}\label{eq:combined_contributions}
 \sum_{j=1}^P \int_{\M}\chi v^{(1,j)}v^{(2,j)}\overline{v^{(3,j)}v^{(4,j)}} \\
 \to \int_{I\times [0,T]} c(t)\chi(\gamma(t))  q(x_1,\gamma(t)) e^{i\lambda_1(x_1+it)}   e^{\ol{i\lambda_2(x_1+it)}}  \abs{a_0}^4|_{(x_1,\gamma(t))}  \,dt \,dx_1
\end{multline}
as $\tau\to \infty$. This is because the intersection of all the sets $\gamma(I_j)$, $j=1,\ldots, P$, contains only the point $p$ (which has measure zero). Recall also that $\chi$ is supported in $U$. 

Let us next consider the integral
\begin{equation}\label{eq:cross_terms}
 \tau^{\frac{n-2}{2}}\int_{\M} \chi\s  q\s  e^{i\lambda_1x_1}e^{-i\lambda_2x_1} v^{(1,j_1)}v^{(2,j_2)}\overline{v^{(3,j_3)}v^{(4,j_4)}} \,dV,
\end{equation}
where $j_a\neq j_b$ for some $a,b=1,\ldots,4$. We argue that this integral is $O(\tau^{-\frac 12})$ and thus negligible. Since $j_a\neq j_b$, we have $t_{j_a}\neq t_{j_b}$ by \eqref{eq:intersection_times}. If $\dot\gamma(t_{j_a})$ would be proportional to $\dot\gamma(t_{j_b})$, then the geodesic $\gamma$ would be a loop by uniqueness of geodesics. Since $\gamma$ is non-tangential, this can not be the case and thus
\[
 \dot\gamma(t_{j_a}) \text{ is not proportional to } \dot\gamma(t_{j_b}).
\]
It follows from the last two conditions of \eqref{eq:phase_properties} that
\[
 \mathrm{Im}(\nabla^2 (\psi^{(j_a)}+\psi^{(j_b)}))(p) > 0
\]
and then using the third condition of \eqref{eq:phase_properties} again we see that
\begin{equation}\label{eq:exponential_decay_in_all}
 \mathrm{Im}(\nabla^2 (\psi^{(j_1)}+\psi^{(j_2)}+\psi^{(j_3)}+\psi^{(j_4)}))(p) > 0.
\end{equation}
It follows that the integrand in \eqref{eq:cross_terms} is exponentially localized not only to the $n-2$ directions transversal to the graph of $\gamma$, but to all $n-1$ directions of $M_0$ pointing away from $p$.  
Consequently, if $U$ is small enough, 
\begin{equation*}
 \left|\tau^{\frac{n-2}{2}}\int_{I\times U} \chi\s  q\s  e^{i\lambda_1x_1}e^{-i\lambda_2x_1} v^{(1,j_1)}v^{(2,j_2)}\overline{v^{(3,j_3)}v^{(4,j_4)}} \,dV\right|\leq C \tau^{\frac{n-2}{2}}\int_{\R^{n-1}}e^{-c\tau |x|^2} \,dx=O(\tau^{-\frac 12}).
\end{equation*}
Here we used that $\chi$ is supported in $U$ and $c$ is some positive constant resulting from \eqref{eq:exponential_decay_in_all}. We note that no stationary phase argument was used here.

The above arguments and estimates hold if the neighborhood $U$ of $p\in \gamma([0,T])$ was small enough. Let us next consider choosing for each $p\in \gamma([0,T])$ a neighborhood $U_p$ such that the above arguments and estimates hold.
Since $\gamma([0,T])$ is compact, we may choose a finite number of sets $U_p$ such that these sets cover a neighborhood of $\gamma([0,T])$ in $\M$. Denote the cover of $\gamma([0,T])$ constituting of these sets by $\{U_\alpha\}$. Let $\{\chi_\alpha\}$ be a partition of unity subordinate $\{U_\alpha\}$. 

We have that the quasimodes $v^1_{\tau+i\lambda_1}, \ldots, v^4_{\tau}$ are supported in $I\times \cup_\alpha U_\alpha$. Combining the above facts and estimates,  we have
\begin{multline}\label{eq:ii_combined}
 0=\int_{\widetilde M} q\s u^1 \cdots u^4 \,dV = \sum_\alpha \int_{I\times U_\alpha} \chi_\alpha q\s u^1 \cdots u^4 \s dV \\
 = \sum_\alpha \sum_{j=1}^{P_\alpha} \tau^{\frac{n-2}{2}}\int_{I\times U_\alpha} \chi_\alpha\s  q\s  e^{i\lambda_1x_1}e^{-i\lambda_2x_1} v^{(\alpha,1,j)}v^{(\alpha,2,j)}\overline{v^{(\alpha,3,j)}v^{(\alpha,4,j)}} \,dV+O(\tau^{-1})+O(\tau^{-\frac 12}).
\end{multline}
Here the terms $O(\tau^{-1})$ and  $O(\tau^{-\frac 12})$ correspond to $R_\tau$ and terms of the type \eqref{eq:cross_terms} respectively. Here, for $k=1,\ldots,4$ and $j=1,\ldots, P_\alpha$, we have denoted by
\[
 v^{(\alpha,k,j)},
\]
the function $v^{(k,j)}$ in the presentation \eqref{eq:gaussian_beam_finite_sum} with respect to the open set $U_\alpha$. Taking the limit $\tau\to\infty$ of \eqref{eq:ii_combined}, then gives 
\begin{multline*}
 0=\sum_\alpha \sum_{j=1}^{P_\alpha} \int_{I\times [0,T]} \chi_\alpha (\gamma(t))   q(x_1,\gamma(t))c(t) e^{i\lambda_1(x_1+it)}   e^{\ol{i\lambda_2(x_1+it)}} \abs{a_0(x_1,\gamma(t))}^4 \,dx_1\,dt \\
 = \sum_\alpha \int_{I\times [0,T]} \chi_\alpha (\gamma(t))  q(x_1,\gamma(t))c(t) e^{i\lambda_1(x_1+it)}   e^{\ol{i\lambda_2(x_1+it)}} \abs{a_0(x_1,\gamma(t))}^4 \,dx_1 \,dt \\
 =\int_{I\times [0,T]} q(x_1,\gamma(t))c(t) e^{i\lambda_1(x_1+it)}   e^{\ol{i\lambda_2(x_1+it)}} \abs{a_0(x_1,\gamma(t))}^4 \,dx_1 \,dt,
\end{multline*}
where we used \eqref{eq:limit_identity} and \eqref{eq:combined_contributions}. The rest of the proof is as in Step 3 above. 
\end{proof}

\begin{proof}[Proof of Corollary \ref{cor:CTA}]
The only thing to check is that also in the case where $q_j$ is only continuous,  $\Lambda_{q_1}=\Lambda_{q_2}$ implies
 \begin{equation*}
	\int_M q \s u_1 \cdots u_{m+1} \,dV = 0,
\end{equation*}
for $u_k$ solving $\Delta_gu_k=0$. This however was verified for $q_j\in L^{n/2+\eps}$ in \cite{nurminen2023determining} in $\R^n$. The  argument on Riemannian manifolds is the same. Thus Corollary \ref{cor:CTA} follows from Theorem \ref{thm:density_of_4}.
\end{proof}

\begin{Remark} 
 The case $m=2$ of Theorem \ref{thm:density_of_4} corresponds to the integral identity of the Calder\'on problem for the linear Schr\"odinger operator. In this case one obtains 
\begin{align*}
  0&=\int_{I\times [0,T]} \tilde{q}(x,t)e^{i \lambda (x_1+it)} \,dt \,dx_1
 \end{align*}
 by arguments analogous to those in the proof of Theorem \ref{thm:density_of_4} above. Here $\tilde q$ is a smooth multiple of the difference of the unknown potentials of the problem. 
 By differentiating in $\lambda$ at $\lambda=0$ it follows that
 \[
 0=\int_{I\times [0,T]} z^l\s \tilde{q}(z) \,dt \,dx_1
\]
for $l\in \N\cup \{0\}$. Here $z=x_1+it$ and $\tilde{q}(z)=\tilde{q}(x,t)$.  Stone-Weierstrass does not apply in this case since the algebra $\mathrm{span}\{z^l\}$ is not closed under complex conjugation (i.e.\ polynomials involving $\bar{z}$ are missing). Indeed, if there were a sequence $q_n(z)$ of polynomials of $z$ converging uniformly to $\tilde q(z)$, then since $\dbar q_n = 0$ it would follow that $\dbar \tilde{q} = 0$. 
Thus $\tilde q$ would have to be holomorphic. It follows that the method of proof of Theorem \ref{thm:density_of_4} cannot be applied for the Calder\'on problem for the linear Schr\"odinger operator.

The case $m=3$ of Theorem \ref{thm:density_of_4} corresponds to the Calder\'on problem for the equation $\Delta_gu+qu^2=0$ on CTA manifolds. Recovery of $q$ from the DN map in this case was shown in \cite{FLL2023}. We expect that the Stone-Weierstrass theorem can be used for the $qu^2$ nonlinearity as well, but proving this would require constructing somewhat different solutions to the ``overlinearized'' equations in \cite{FLL2023}. For this reason, we did not consider the case $m=3$  in Corollary \ref{cor:CTA} in this paper.
\end{Remark}

\bibliographystyle{alpha}
\bibliography{ref_stoneweierstrass}

\end{document}